\documentclass[12pt]{article}

\usepackage{amsmath,amssymb,bm,amsthm,mathrsfs,amscd}
\usepackage{here}
\usepackage[pdftex]{graphicx}
\usepackage{tikz}
\usepackage{color}
\theoremstyle{definition}
\newtheorem{Def}{Def}[section]
\newtheorem{Defi}[Def]{Definition}
\newtheorem{Them}[Def]{Theorem}
\newtheorem{Lem}[Def]{Lemma}
\newtheorem{Cor}[Def]{Corollary}
\newtheorem{Prop}[Def]{Proposition}
\newtheorem{Examp}[Def]{Example}

\DeclareMathOperator{\ord}{ord}

\numberwithin{equation}{section}

\title{Freeness for restriction arrangements of the extended Shi and Catalan arrangements}
\author{Norihiro Nakashima\thanks{Department of Mathematics, Nagoya Institute of Technology, Aichi, 466-8555, Japan. Email: nakashima@nitech.ac.jp}\quad
and\quad Shuhei Tsujie\thanks{Department of Mathematics, Hokkaido University of Education, Asahikawa, Hokkaido 070-8621, Japan. Email: tsujie.shuhei@a.hokkyodai.ac.jp}}

\date{}
\begin{document}

\maketitle
\begin{center}
  \emph{Dedicated to Professor Mutsumi Saito on his 63rd birthday}
\end{center}
\begin{abstract}
The extended Shi and Catalan arrangements are well investigated arrangements. In this paper, we prove that the cone of the extended Catalan arrangement of type A is always hereditarily free, while we determine the dimension in which the cone of the extended Shi arrangement of type A is hereditarily free. For this purpose, using digraphs, we define a class of arrangements which is closed under restriction, and which contains the extended Shi and Catalan arrangements. We also characterize the freeness for the cone of this arrangement by graphical conditions.

\vspace{2mm}

\end{abstract}

\section{Introduction}
Let $\mathbb{K}$ be a field of characteristic zero.
A (hyperplane) $\ell$-arrangement, or simply an arrangement, is a finite set of affine hyperplanes in an $\ell$-dimensional vector space, or affine space, over $\mathbb{K}$.
An arrangement $\mathscr{A}$ is said to be central if all hyperplanes in $\mathscr{A}$ contain the origin.
Let $V$ be a finite set of $\ell$ elements, let $S$ be the polynomial ring over $\mathbb{K}$ in $\ell$ variables indexed by $V$, and let $\partial_i\ (i\in V)$ be the partial derivative.
Let $\alpha_H$ be a defining polynomial of $H\in\mathscr{A}$, that is, $H=\{\bm{x}\in\mathbb{K}^{V}\mid \alpha_H(\bm{x})=0\}$.
Then a central $\ell$-arrangement $\mathscr{A}$ is free if the $S$-module
\begin{align*}
D(\mathscr{A}):=\left\{\theta\in\sum_{i\in V}S\partial_i\,\middle|\,\theta(\alpha_H)\in\alpha_H S\ {\rm for\ all}\ H\in\mathscr{A}\right\}
\end{align*}
is a free $S$-module.
When $\mathscr{A}$ is a non-central arrangement, we consider
the freeness for the cone $c\mathscr{A}$ instead of $\mathscr{A}$, where the $(\ell+1)$-arrangement $c\mathscr{A}$ is defined by
\begin{align*}
c\mathscr{A}:=\{\{z=0\}\}\cup\{\{\alpha_H(\bm{x})+\alpha_H(\bm{0})(z-1)=0\}\mid H\in\mathscr{A}\}.
\end{align*}

For any arrangement $\mathscr{A}$ (including a non-central arrangement), let
\begin{align*}
L(\mathscr{A}):=\left\{\bigcap_{H\in\mathscr{B}}H\,\middle|\, \mathscr{B}\subseteq\mathscr{A},\ \bigcap_{H\in\mathscr{B}}H\neq \emptyset\right\},
\end{align*}
and we define the restriction arrangement $\mathscr{A}^X$ by
\begin{align*}
\mathscr{A}^X:=\left\{H\cap X\,\middle|\, H\in\mathscr{A},\,X\not\subseteq H,\,H\cap X\neq \emptyset\right\}
\end{align*}
for $X\in L(\mathscr{A})$.
A central arrangement $\mathscr{A}$ is said to be hereditarily free if all restriction arrangements $\mathscr{A}^X$ are free for all $X\in L(\mathscr{A})$.
There is a famous conjecture of freeness for restriction arrangements given by Orlik \cite{Orlik-intro-arr}, which asserts that if $\mathscr{A}$ is a free arrangement, then $\mathscr{A}^H$ is free for any $H\in\mathscr{A}$.
Edelman and Reiner \cite{Edelman-Reiner-orlikconje} found a counterexample of this conjecture.
After that several investigations are interested in hereditarily free arrangements.
Orlik and Terao \cite{Orlik-Terao-hered-free} proved that all Coxeter arrangements are hereditarily free, while Douglass \cite{Douglass} gave another proof of this result using a Lie theoretic argument.
Later Hoge and R\"{o}hrle \cite{Hoge-Roehrle-hered-free} proved that the finite complex reflection arrangements are hereditarily free.

At the same time, the freeness for deformations of the braid arrangement (or Coxeter arrangement of Type A) $\mathcal{A}_{\ell-1}:=\{\{x_i-x_j=0\}\mid i,j\in V,\ i\neq j\}$ is a central topic in the study of free arrangements.
Some remarkable deformations are investigated, using a digraph $G=(V,E)$, where $E\subseteq \{(i,j)\mid i,j\in V, i\neq j \}$ is a set of directed edges.
For $i,j\in V\ (i\neq j)$, we define
\begin{align*}
\varepsilon_G(i,j):=
\begin{cases}
1&{\rm if}\ (i,j)\in E,\\ 0&{\rm if}\ (i,j)\not\in E.
\end{cases}
\end{align*}
In particular, the arrangement
\begin{align}\label{eq-derom-by-Athana}
\mathscr{A}_m(G):=\left\{\{x_i-x_j=c\}\,\middle|\,i,j\in V,i\neq j,-m-\varepsilon_G(i,j)\leq c\leq m+\varepsilon_G(j,i)\right\},
\end{align}
includes two important arrangements; the extended Catalan arrangement 
\begin{align*}
\mathcal{C}_{\ell-1}(m):=\{\{x_i-x_j=c\}\mid i,j\in V,i\neq j,-m\leq c\leq m\}
\end{align*}
of type A and the extended Shi arrangement 
\begin{align*}
\mathcal{S}_{\ell-1}(m):=\{\{x_i-x_j=c\}\mid i,j\in\{1,\dots,\ell\},i<j,-m\leq c\leq m+1\}
\end{align*}
of type A.
Athanasiadis \cite{Athanasiadis2} proved that $c\mathcal{S}_{\ell-1}(m)$ is free, while Yoshinaga \cite{Yoshinaga-ER} extended this result to more general settings including $c\mathcal{S}_{\ell-1}(m)$ and $c\mathcal{C}_{\ell-1}(m)$.
When $m=0$, Athanasiadis \cite{Athanasiadis3} conjectured that $c\mathscr{A}_m(G)$ is free if and only if there exists a total order $\preceq$ on $V$ such that $G$ satisfies
\begin{itemize}
\item[(A1)] If $i\prec k,j\prec k$ and $(i,j)\in E$, then $(i,k)\in E$ or $(k,j)\in E$.
\item[(A2)] If $i\prec k,j\prec k$ and $(i,k)\in E, (k,j)\in E$, then $(i,j)\in E$
\end{itemize}
This conjecture was proved for any $m\in\mathbb{Z}_{\geq 0}$.
The proof of ``if'' part was given by Abe, Nuida and Numata \cite{Abe-Nuida-Numata} and the proof of ``only if'' part was given by Abe \cite{Abe}.
There are several studies characterizing freeness of deformations of the Coxeter arrangements of types A and B in terms of graphs (See \cite{abe2017freeness-joctsa, abe2021vertex-weighted, bailey????inductively-p, bailey1997tilings, edelman1994free-mz, suyama2019signed-dm, suyama2019vertex-weighted-dcg, torielli2020freeness-tejoc, tran2022mat-free, wang2022free-gac}).

In this paper, we aim to prove that $c\mathcal{C}_{\ell-1}(m)$ is hereditarily free and that $c\mathcal{S}_{\ell-1}(m)$ is hereditarily free if and only if $\ell\leq 5$.
For this purpose, it is important to know a closed class under restriction which contains $\mathcal{C}_{\ell-1}(m)$ and $\mathcal{S}_{\ell-1}(m)$.
The class of arrangements $\mathscr{A}_m(G)$ contains $\mathcal{C}_{\ell-1}(m)$ and $\mathcal{S}_{\ell-1}(m)$, but unfortunately is not closed under restriction.
So we define the following arrangement $\mathscr{A}(\bm{n},G)$ which contains $\mathscr{A}_m(G)$, and prove that the class of arrangement $\mathscr{A}(\bm{n},G)$ is closed under restriction. Note that this class may not be a minimal class which is closed under restriction and contains $\mathcal{C}_{\ell-1}(m)$ and $\mathcal{S}_{\ell-1}(m)$.
\begin{Defi}
Let $\ell\geq 2$. Let $\bm{n}=(n_i)_{i\in V}\in \mathbb{Z}_{\geq 0}^{V}$ be a tuple indexed by $V$ and let $G=(V,E)$ be a digraph.
We define the arrangement
\begin{align*}
\mathscr{A}(\bm{n},G):=\left\{\{x_i-x_j=c\}\,\middle|\,i,j\in V,i\neq j,c\in\mathbb{Z},-n_i-\varepsilon_G(i,j)\leq c\leq n_j+\varepsilon_G(j,i)\right\}.
\end{align*}
\end{Defi}
To prove that the class is closed under restriction, we define the contraction $(\bm{n}^{H},G^{H})$ of a pair $(\bm{n},G)$ so that $\mathscr{A}(\bm{n},G)^H=\mathscr{A}(\bm{n}^{H},G^{H})$ up to affine equivalence.
The contraction is different from the ordinary vertex contraction of digraphs.
In addition, if $E=\emptyset$, then $\mathscr{A}((m,\dots,m),G)=\mathcal{C}_{\ell-1}(m)$, while if $V=\{1,\dots,\ell\}$ and $E=\{(j,i)\mid i<j\}$, then $\mathscr{A}((m,\dots,m),G)=\mathcal{S}_{\ell-1}(m)$.

We also generalize the freeness characterization by Athanasiadis for $c\mathscr{A}_m(G)$ to $c\mathscr{A}(\bm{n},G)$. In other words, we prove that $c\mathscr{A}(\bm{n},G)$ is free if and only if there exists a total order $\preceq$ on $V$ such that $G$ satisfies (A1) and (A2).
In the proof, we use several results for characteristic polynomials, signed graphs, and multi-arrangements.
After that we prove that $c\mathscr{A}(\bm{n},G)$ is hereditarily free if $E=\emptyset$.
In addition, when $V=\{1,\dots,\ell\}$ and $E=\{(j,i)\mid i<j\}$, we prove that $c\mathscr{A}(\bm{n},G)$ is hereditarily free if and only if $\ell\leq 5$, using the results above.
In particular, we have that $c\mathcal{C}_{\ell-1}(m)$ is hereditarily free, while $c\mathcal{S}_{\ell-1}(m)$ is hereditarily free if and only if $\ell\leq 5$.

The organization of this paper is as follows. In Section \ref{sec-A(n,G)-res}, we prove that the class of arrangements $\mathscr{A}(\bm{n},G)$ is closed under restriction.
In Section \ref{sec-free-A(n,G)}, we characterize the freeness for $c\mathscr{A}(\bm{n},G)$.
Finally in Section \ref{sec-hereditarily-free}, we prove that $c\mathcal{C}_{\ell-1}(m)$ is hereditarily free and that $c\mathcal{S}_{\ell-1}(m)$ is hereditarily free if and only if $\ell\leq 5$.

\section{The arrangements $\mathscr{A}(\bm{n},G)$ and their restrictions}\label{sec-A(n,G)-res}
In this paper, we often denote by $\{x_i-x_j=c\}$ the hyperplane $\{\bm{x}\in \mathbb{K}^V\mid x_i-x_j=c\}$.
Let $\ell\geq 2$. Let $\bm{n}=(n_i)_{i\in V}\in \mathbb{Z}_{\geq 0}^{V}$ be a tuple indexed by $V$ and let $G=(V,E)$ be a digraph.
We recall that the definition of $\mathscr{A}(\bm{n},G)$ is
\begin{align}
\mathscr{A}(\bm{n},G)=\left\{\{x_i-x_j=c\}\,\middle|\,i,j\in V,i\neq j,c\in\mathbb{Z},-n_i-\varepsilon_G(i,j)\leq c\leq n_j+\varepsilon_G(j,i)\right\}.\label{eq-defAnG}
\end{align}
Since $\{x_i-x_j=c\}=\{x_j-x_i=-c\}$, the arrangement $\mathscr{A}(\bm{n},G)$ also has the following description:
\begin{align}\label{eq-AnG3}
\mathscr{A}(\bm{n},G)=\left\{\{x_i-x_j=c\}\,\middle|\, i,j\in V,i\neq j,c\in\mathbb{Z},0\leq c\leq n_j+\varepsilon_G(j,i)\right\}.
\end{align}
We note that the number of hyperplanes in $\mathscr{A}(\bm{n},G)$ which include $x_i$ and $x_j$ in the defining polynomials is $n_i+n_j+\varepsilon_G(i,j)+\varepsilon_G(j,i)+1$.
\begin{Examp}\label{ex-A(n,G)}
Let $V=\{1,2,3\}$, $E=\{(2,1),(2,3)\}$ and $\bm{n}=(0,0,1)$. Then
\begin{align*}
\mathscr{A}(\bm{n},G)=\left\{
\begin{array}{cccc}
\{x_1-x_2=0\},&\{x_1-x_2=1\},&\{x_1-x_3=0\},&\{x_1-x_3=1\},\\
\{x_2-x_3=-1\},&\{x_2-x_3=0\},&\{x_2-x_3=1\}& 
\end{array}
\right\}.
\end{align*}
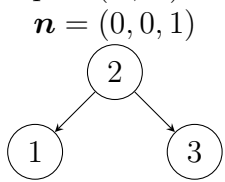
\begin{figure}[H]
 \centering
\caption{The pair $(\bm{n},G)$ in Example \ref{ex-A(n,G)}}

\begin{tabular}{c}
$\bm{n}=(0,0,1)$\\
\begin{tikzpicture}[node distance={15mm}, main/.style = {draw, circle}] 
\node[main] (1) {$2$}; \node[main] (2) [below left of=1] {$1$}; \node[main] (3) [below right of=1] {$3$};
\draw[->,>=stealth] (1) to (2); \draw[->,>=stealth] (1) to (3);
\end{tikzpicture}
\end{tabular}
\end{figure}
\end{Examp}

Two arrangements $\mathscr{A}_1$ and $\mathscr{A}_2$ in affine spaces $A_1$ and $A_2$ are said to be affinely equivalent if there exists an affine isomorphism $\phi:A_1\rightarrow A_2$ such that $\mathscr{A}_2=\phi(\mathscr{A}_1):=\{\phi(H)\mid H\in\mathscr{A}_1\}$.
For $H\in\mathscr{A}(\bm{n},G)$, we define the contraction $(\bm{n}^{H},G^{H})$ of a pair $(\bm{n},G)$ so that $\mathscr{A}(\bm{n},G)^H=\mathscr{A}(\bm{n}^{H},G^{H})$ up to affine equivalence.
This means that the class of arrangements $\mathscr{A}(\bm{n},G)$ is closed under restriction.
The proof will appear later in Theorem \ref{thm-restAnGtoH}.
\begin{Defi}\label{def-contract-H}
Let $(\bm{n},G)$ be a pair of a tuple $\bm{n}=(n_i)_{i\in V}\in \mathbb{Z}_{\geq 0}^{V}$ and a digraph $G=(V,E)$.
For $H\in\mathscr{A}(\bm{n},G)$, we can write $H=\{x_s-x_t=w\}$ for some $0\leq w\leq n_t+\varepsilon_G(t,s)$ by the equality \eqref{eq-AnG3}.
Then we define the contraction $G^{H}:=(V^{H},E^{H})$ as follows.
\begin{itemize}
\item $V^{H}:=\left(V\setminus\{s,t\}\right)\cup\{u\}$,
where $u$ is a new vertex.
\item For $i\in V^H\setminus\{u\}$,
\begin{align*}
&(i,u)\in E^H\Leftrightarrow
\begin{cases}
(i,t)\in E&\quad{\rm if}\ w>0,\\
(i,s)\in E\ {\rm or}\ (i,t)\in E&\quad{\rm if}\ w=0,
\end{cases}
\\
&(u,i)\in E^H\Leftrightarrow
\begin{cases}
(s,i)\in E&\quad{\rm if}\ n_s+w>n_t,\\
(s,i)\in E\ {\rm or}\ (t,i)\in E&\quad{\rm if}\ n_s+w=n_t,\\
(t,i)\in E&\quad{\rm if}\ n_s+w<n_t.
\end{cases}
\end{align*}
\item For $i,j\in V^H\setminus\{u\}$,
\begin{align*}
(i,j)\in E^H\Leftrightarrow (i,j)\in E.
\end{align*}
\end{itemize}
\end{Defi}
\begin{Defi}
Let $\bm{n}\in\mathbb{Z}_{\geq 0}^{V}$ be a tuple.
For $H\in\mathscr{A}(\bm{n},G)$, we can write $H=\{x_s-x_t=w\}$ for some $0\leq w\leq n_t+\varepsilon_G(t,s)$.
Then we define
\begin{align*}
n^{H}_i:=
\begin{cases}
n_i&\quad{\rm if}\ i\in V^{H}\setminus\{u\},\\
\max\{n_s+w,n_t\}&\quad{\rm if}\ i=u,
\end{cases}
\end{align*}
and $\bm{n}^H:=(n_i^H)_{i\in V^H}$.
\end{Defi}

\noindent
{\it Remark.}
\begin{itemize}
\item[(1)] Let $H\in\mathscr{A}(\bm{n},G)$, and we write $H=\{x_s-x_t=w\}$ for some $0\leq w\leq n_t+\varepsilon_G(t,s)$. Then $w=0$ and $n_s=n_t$ if and only if the contraction $G^H$ is the same as the digraph obtained from ordinary vertex contraction of vertices $s$ and $t$. Meanwhile in this case, $n_u^H=n_s=n_t$.
\item[(2)] Let $H=\{x_s-x_t=n_t+\varepsilon_G(t,s)\}\in\mathscr{A}(\bm{n},G)$. If $(t,s)\in E$, then $H=\{x_s-x_t=n_t+1\}$ while $\mathscr{A}(\bm{n},G)\setminus\{H\}=\mathscr{A}(\bm{n},(V,E\setminus\{(t,s)\}))$.
So deleting the hyperplane $H$ corresponds to deleting the edge $(t,s)$.
In this case, $n_u^H=n_s+n_t+1$, the edges in $E^H$ arriving at the new vertex $u$ are obtained from the edges in $E$ arriving at $t$, while the edges in $E^H$ departing from $u$ are obtained from the edges in $E$ departing from $s$.
\end{itemize}
\begin{Examp}\label{ex-(n,G)-contraction}
(1) Let $V=\{1,2,3,4\}$, $E=\{(2,1),(3,1),(4,1),(4,2),(4,3)\}$ and $\bm{n}=(1,0,0,1)$.
Let $H=\{x_1-x_4=0\}\in\mathscr{A}(\bm{n},G)$.
Then by writing $u=4$, we have that $\bm{n}^H=(n_2,n_3,n_4)=(0,0,1)$ and $G^H=(\{2,3,4\},\{(2,4),(4,2),(3,4),(4,3)\})$.
\begin{figure}[H]
 \centering
\caption{The pairs $(\bm{n},G)$ and $(\bm{n}^H,G^H)$ in Example \ref{ex-(n,G)-contraction} (1)}

\begin{tabular}{c@{\hspace{30mm}}c}
$\bm{n}=(1,0,0,1)$&$\bm{n}^H=(0,0,1)$\\
\begin{tikzpicture}[node distance={15mm}, main/.style = {draw, circle}] 
\node[main] (1) {$1$}; \node[main] (2) [below left of=1] {$2$}; \node[main] (3) [below right of=1] {$3$}; \node[main] (4) [below right of=2] {$4$};
\draw[->,>=stealth] (2) -- (1); \draw[->,>=stealth] (3) -- (1); \draw[->,>=stealth] (4) -- (1); \draw[->,>=stealth] (4) -- (2); \draw[->,>=stealth] (4) -- (3); \end{tikzpicture}
&\begin{tikzpicture}[node distance={15mm}, main/.style = {draw, circle}] 
\node[main] (1) {$4$}; \node[main] (2) [below left of=1] {$2$}; \node[main] (3) [below right of=1] {$3$};
\draw[->,>=stealth] (2) to [out=65,in=200,looseness=0.7] (1); \draw[->,>=stealth] (1) to [out=245,in=25,looseness=0.7] (2); \draw[->,>=stealth] (1) to [out=335,in=115,looseness=0.7] (3); \draw[->,>=stealth] (3) to [out=155,in=295,looseness=0.7] (1);
\end{tikzpicture}
\end{tabular}
\end{figure}

(2) Let $V=\{1,2,3,4\}$, $E=\{(j,i)\mid1\leq i<j\leq 4\}$ and $\bm{n}=\bm{0}$.
Let $H=\{x_1-x_3=n_3+\varepsilon_G(3,1)\}=\{x_1-x_3=1\}\in\mathscr{A}(\bm{n},G)$.
Then by writing $u=3$, we have that $\bm{n}^H=(n_2,n_3,n_4)=(0,1,0)$ and $G^H=(\{2,3,4\},\{(4,2),(4,3)\})$.

\begin{figure}[H]
 \centering
\caption{The pairs $(\bm{n},G)$, $(\bm{n}^{H},G^{H})$ in Example \ref{ex-(n,G)-contraction} (2)}

\begin{tabular}{c@{\hspace{10mm}}c}
$\bm{n}=(0,0,0,0)$&$\bm{n}^{H}=(0,1,0)$\\
\begin{tikzpicture}[node distance={15mm}, main/.style = {draw, circle}] 
\node[main] (1) {$1$}; \node[main] (2) [below left of=1] {$2$}; \node[main] (3) [below right of=1] {$3$}; \node[main] (4) [below right of=2] {$4$};
\draw[->,>=stealth] (2) -- (1); \draw[->,>=stealth] (3) -- (1); \draw[->,>=stealth] (3) -- (2); \draw[->,>=stealth] (4) -- (1); \draw[->,>=stealth] (4) -- (2); \draw[->,>=stealth] (4) -- (3);
\end{tikzpicture}
&\begin{tikzpicture}[node distance={15mm}, main/.style = {draw, circle}] 
\node[main] (1) {$3$}; \node[main] (2) [below left of=1] {$2$}; \node[main] (4) [below right of=2] {$4$};
\draw[->,>=stealth] (4) -- (1); \draw[->,>=stealth] (4) -- (2);
\end{tikzpicture}
\end{tabular}
\end{figure}
\end{Examp}
For integers $a,b\in\mathbb{Z}$ with $a\leq b$, we define $[a,b]:=\{c\in\mathbb{Z}\mid a\leq c\leq b\}$.
\begin{Lem}\label{lem-max-ninu}
Let $H\in\mathscr{A}(\bm{n},G)$, and we write $H=\{x_s-x_t=w\}$ for some $0\leq w\leq n_t+\varepsilon_G(t,s)$.
Then for $i\in V^H\setminus\{u\}$, we have
\begin{itemize}
\item[(1)] $\max\{n_i-w+\varepsilon_{G}(i,s),\ n_i+\varepsilon_{G}(i,t)\}=n_i^H+\varepsilon_{G^H}(i,u)$,
\item[(2)] $\max\{n_s+w+\varepsilon_{G}(s,i),\ n_t+\varepsilon_{G}(t,i)\}=n_u^H+\varepsilon_{G^H}(u,i)$,
\item[(3)]
\begin{multline*}
[-n_i+w-\varepsilon_{G}(i,s),n_s+w+\varepsilon_{G}(s,i)]\cup[-n_i-\varepsilon_{G}(i,t),n_t+\varepsilon_{G}(t,i)]=\\
[-n_i^H-\varepsilon_{G^H}(i,u),n_u^H+\varepsilon_{G^H}(u,i)].
\end{multline*}
\end{itemize}
\end{Lem}
\begin{proof}
(1)\ We prove the assertion by a case study of (1-a) $w>0$ and (1-b) $w=0$.

(1-a) In the case $w>0$, since $n_i-w<n_i$ and
$\varepsilon_{G}(i,s),\varepsilon_{G}(i,t)\in\{0,1\}$, we have
\begin{align*}
n_i-w+\varepsilon_{G}(i,s)\leq n_i\leq n_i+\varepsilon_{G}(i,t).
\end{align*}
By the definitions, $n_i^H=n_i$ and
$\varepsilon_{G^H}(i,u)=\varepsilon_{G}(i,t)$. Therefore
\begin{align*}
\max\{n_i-w+\varepsilon_{G}(i,s),\ n_i+\varepsilon_{G}(i,t)\}
=n_i+\varepsilon_{G}(i,t)
=n_i^H+\varepsilon_{G^H}(i,u).
\end{align*}

(1-b) In the case $w=0$, by the definitions, we have $n_i^H=n_i=n_i-w$ and
$\varepsilon_{G^H}(i,u)=
\max\left\{\varepsilon_{G}(i,s),\varepsilon_{G}(i,t)\right\}$. Therefore
\begin{align*}
\max\{n_i-w+\varepsilon_{G}(i,s),\ n_i+\varepsilon_{G}(i,t)\}=
n_i^H+\max\left\{\varepsilon_{G}(i,s),\varepsilon_{G}(i,t)\right\}
=n_i^H+\varepsilon_{G^H}(i,u).
\end{align*}

(2) We prove the assertion by a case study of (2-a) $n_s+w>n_t$,
(2-b) $n_s+w=n_t$ and (2-c) $n_s+w<n_t$.

(2-a) In the case $n_s+w>n_t$, since
$\varepsilon_{G}(s,i),\varepsilon_{G}(t,i)\in\{0,1\}$, we have
\begin{align*}
n_t+\varepsilon_{G}(t,i)\leq n_s+w\leq n_s+w+\varepsilon_{G}(s,i).
\end{align*}
By the definitions, $n_u^H=n_s+w$ and
$\varepsilon_{G^H}(u,i)=\varepsilon_{G}(s,i)$. Therefore
\begin{align*}
\max\{n_s+w+\varepsilon_{G}(s,i),\ n_t+\varepsilon_{G}(t,i)\}
=n_s+w+\varepsilon_{G}(s,i)=n_u^H+\varepsilon_{G^H}(u,i).
\end{align*}

(2-b) In the case $n_s+w=n_t$, by the definitions, we have $n_u^H=n_s+w=n_t$ and $\varepsilon_{G^H}(u,i)=\max\left\{\varepsilon_{G}(s,i),\varepsilon_{G}(t,i)\right\}$. Therefore
\begin{align*}
\max\{n_s+w+\varepsilon_{G}(s,i),\ n_t+\varepsilon_{G}(t,i)\}
=n_u^H+\max\left\{\varepsilon_{G}(s,i),\varepsilon_{G}(t,i)\right\}
=n_u^H+\varepsilon_{G^H}(u,i).
\end{align*}

(2-c) In the case $n_s+w<n_t$, since $\varepsilon_{G}(s,i),\varepsilon_{G}(t,i)\in\{0,1\}$, we have
\begin{align*}
n_s+w+\varepsilon_{G}(s,i)\leq n_t\leq n_t+\varepsilon_{G}(t,i).
\end{align*}
By the definitions, $n_u^H=n_t$ and $\varepsilon_{G^H}(u,i)=\varepsilon_{G}(t,i)$. Therefore
\begin{align*}
\max\{n_s+w+\varepsilon_{G}(s,i),\ n_t+\varepsilon_{G}(t,i)\}
=n_t+\varepsilon_{G}(t,i)=n_u^H+\varepsilon_{G^H}(u,i).
\end{align*}

(3) In general, for $a,b,c,d\in\mathbb{Z}$ with $a\leq b$ and $c\leq d$, the equality $[a,b]\cup[c,d]=[\min\{a,c\},\max\{b,d\}]$ holds if and only if $d-a\geq -1$ and $b-c\geq -1$.
Hence
\begin{align}
&[-n_i+w-\varepsilon_{G}(i,s),n_s+w+\varepsilon_{G}(s,i)]\cup[-n_i-\varepsilon_{G}(i,t),n_t+\varepsilon_{G}(t,i)]\notag\\
=\,&\left[\min\{-n_i+w-\varepsilon_{G}(i,s),-n_i-\varepsilon_{G}(i,t)\},\max\{n_s+w+\varepsilon_{G}(s,i),n_t+\varepsilon_{G}(t,i)\}\right]\label{eq-interval-abcapcd}
\end{align}
holds if and only if $n_t+\varepsilon_{G}(t,i)-(-n_i+w-\varepsilon_{G}(i,s))\geq -1$ and $n_s+w+\varepsilon_{G}(s,i)-(-n_i-\varepsilon_{G}(i,t))\geq -1$.
Now $n_s+w+\varepsilon_{G}(s,i)-\left(-n_i-\varepsilon_{G}(i,t)\right)\geq -1$ is clear.
By the assumption $0\leq w\leq n_t+\varepsilon_G(t,s)$, we have
\begin{align*}
n_t+\varepsilon_{G}(t,i)-(-n_i+w-\varepsilon_{G}(i,s))
&\geq n_t+\varepsilon_{G}(t,i)+n_i-n_t-\varepsilon_{G}(t,s)+\varepsilon_{G}(i,s)\\
&=\varepsilon_{G}(t,i)+n_i-\varepsilon_{G}(t,s)+\varepsilon_{G}(i,s)\\
&\geq -1.
\end{align*}
Therefore the equality \eqref{eq-interval-abcapcd} holds, and the assertion follows from (1) and (2).
\end{proof}
\begin{Them}\label{thm-restAnGtoH}
Let $G$ be a digraph and let $\bm{n}\in\mathbb{Z}_{\geq 0}^{V}$.
Let $H\in\mathscr{A}(\bm{n},G)$, and we write $H=\{x_s-x_t=w\}$ for some $0\leq w\leq n_t+\varepsilon_G(t,s)$.
Then $\mathscr{A}(\bm{n},G)^H$ and $\mathscr{A}(\bm{n}^{H},G^{H})$ are affinely equivalent.
In particular, the class of arrangements $\mathscr{A}(\bm{n},G)$ is closed under restriction.
\end{Them}
\begin{proof}
We can describe $\mathscr{A}(\bm{n},G)^{H}$ as
\begin{align*}
\mathscr{A}(\bm{n},G)^{H}
=&\left\{\{x_i-x_j=c\}\cap H\,\middle|\,i,j\in V\setminus\{s,t\},i\neq j,-n_i-\varepsilon_G(i,j)\leq c\leq n_j+\varepsilon_G(j,i)\right\}\\
&\cup\left\{\{x_i-x_s=c\}\cap H\,\middle|\,i\in V\setminus\{s,t\},-n_i-\varepsilon_G(i,s)\leq c\leq n_s+\varepsilon_G(s,i)\right\}\\
&\cup\left\{\{x_i-x_t=c\}\cap H\,\middle|\,i\in V\setminus\{s,t\},-n_i-\varepsilon_G(i,t)\leq c\leq n_t+\varepsilon_G(t,i)\right\}.
\end{align*}
Since $x_t=x_s-w$ for $\bm{x}=(x_i)_{i\in V}\in H$, we have
\begin{align*}
&\left\{\{x_i-x_s=c\}\cap H\,\middle|\,i\in V\setminus\{s,t\},-n_i-\varepsilon_G(i,s)\leq c\leq n_s+\varepsilon_G(s,i)\right\}\\
=&\left\{\{x_i-x_t=c+w\}\cap H\,\middle|\,i\in V\setminus\{s,t\},-n_i+w-\varepsilon_G(i,s)\leq c+w\leq n_s+w+\varepsilon_G(s,i)\right\}.
\end{align*}
We note that if $i,j\in V\setminus\{s,t\}$ and $i\neq j$, then $n_i=n_i^H$, $n_j=n^H_j$, $\varepsilon_G(i,j)=\varepsilon_{G^H}(i,j)$, and $\varepsilon_G(j,i)=\varepsilon_{G^H}(j,i)$.
Therefore by Lemma \ref{lem-max-ninu} (3), we have
\begin{align*}
\mathscr{A}(\bm{n},G)^{H}
=&\left\{\{x_i-x_j=c\}\cap H\,\middle|\,i,j\in V\setminus\{s,t\},i\neq j,-n^H_i-\varepsilon_{G^H}(i,j)\leq c\leq n^H_j+\varepsilon_{G^H}(j,i)\right\}\\
&\cup\left\{\{x_i-x_t=c\}\cap H\,\middle|\,i\in V\setminus\{s,t\},-n_i^H-\varepsilon_{G^H}(i,u)\leq c\leq n_u^H+\varepsilon_{G^H}(u,i)\right\}.
\end{align*}
Now let $\phi:H\rightarrow \mathbb{K}^{V^H}$ be the map defined by $\phi(\bm{x})=\bm{y}$, where $y_i:=
\begin{cases}
x_i&\left(i\in V^H\setminus\{u\}\right),\\
x_t&\left(i=u\right).
\end{cases}
$
Then $\phi$ is an affine isomorphism, while
\begin{align*}
&\phi\left(\mathscr{A}(\bm{n},G)^{H}\right)\\
=&\left\{\{\bm{y}\in\mathbb{K}^{V^H}\mid y_i-y_j=c\}\,\middle|\,i,j\in V^H\setminus\{u\},i\neq j,-n^H_i-\varepsilon_{G^H}(i,j)\leq c\leq n^H_j+\varepsilon_{G^H}(j,i)\right\}\\
&\cup\left\{\{\bm{y}\in\mathbb{K}^{V^H}\mid y_i-y_u=c\}\,\middle|\,i\in V^H\setminus\{u\},-n_i^H-\varepsilon_{G^H}(i,u)\leq c\leq n_u^H+\varepsilon_{G^H}(u,i)\right\}\\
=&\mathscr{A}(\bm{n}^H,G^H).
\end{align*}
Hence we conclude that $\mathscr{A}(\bm{n},G)^H$ and $\mathscr{A}(\bm{n}^{H},G^{H})$ are affinely equivalent.
\end{proof}

We prove that if $E$ is the empty set, then $c\mathscr{A}(\bm{n},G)$ is free.
We recall some results which are used in the proof of the assertion.
A derivation $\theta=\sum_{i\in V}f_i\partial_i\in\sum_{i\in V}S\partial_i$ is said to be homogeneous of degree $j$ and write $\deg(\theta)=j$, if $f_i$ is zero or homogeneous of degree $j$ for each $i\in V$.
If a central arrangement $\mathscr{A}$ is free, then there exists a homogeneous basis $\{\theta_1,\dots,\theta_{\ell}\}$ for $D(\mathscr{A})$.
The multi-set $\exp(\mathscr{A})$ of exponents of a free arrangement $\mathscr{A}$ is defined by $\exp(\mathscr{A}):=\left\{\deg(\theta_1),\dots,\deg(\theta_{\ell})\right\}$.
For central arrangements $\mathscr{A}_1$ and $\mathscr{A}_2$ in $\mathbb{K}^{V_1}$ and $\mathbb{K}^{V_2}$, we define the product $\mathscr{A}_1\times\mathscr{A}_2$ by $\mathscr{A}_1\times\mathscr{A}_2:=\{H_1\oplus\mathbb{K}^{V_2}\mid H_1\in\mathscr{A}_1\}\cup\{\mathbb{K}^{V_1}\oplus H_2\mid H_2\in\mathscr{A}_2\}$.
For a finite set $X$, we denote by $\# X$ and $|X|$ the number of elements in $X$.
\begin{Prop}[cf. Proposition 4.28, Example 4.20, and Example 4.32 in \cite{Orlik-Terao}]
\ 
\begin{itemize}
\item[(1)] The product $\mathscr{A}_1\times\mathscr{A}_2$ of central arrangements $\mathscr{A}_1$ and $\mathscr{A}_2$ is free if and only if both $\mathscr{A}_1$ and $\mathscr{A}_2$ are free, while in this case $\exp(\mathscr{A}_1\times\mathscr{A}_2)=\exp(\mathscr{A}_1)\cup\exp(\mathscr{A}_2)$.
\item[(2)] All central $2$-arrangements $\mathscr{A}$ are free with $\exp(\mathscr{A})=\{1,|\mathscr{A}|-1\}$.
\item[(3)] The braid arrangement $\mathcal{A}_{\ell-1}$ is free with $\exp(\mathcal{A}_{\ell-1})=\{0,1,2,\dots,\ell-1\}$.
\end{itemize}
\end{Prop}
\begin{Them}[The addition-deletion theorem, Terao \cite{Terao-add-del}]\label{thm-add-del-terao}
Let $\mathscr{A}$ be a nonempty central $\ell$-arrangement and let $H\in\mathscr{A}$. Then any two of the following statements imply the third:
\begin{itemize}
\item $\mathscr{A}$ is free with $\exp(\mathscr{A})=\{e_1,\dots,e_{\ell-1},e_{\ell}\}$,
\item $\mathscr{A}\setminus\{H\}$ is free with $\exp(\mathscr{A}\setminus\{H\})=\{e_1,\dots,e_{\ell-1},e_{\ell}-1\}$,
\item $\mathscr{A}^H$ is free with $\exp(\mathscr{A}^H)=\{e_1,\dots,e_{\ell-1}\}$.
\end{itemize}
\end{Them}
In this paper, we frequently use Theorem \ref{thm-add-del-terao}.
The class $\mathcal{IF}$ of inductively free arrangements is the smallest class of central arrangements which satisfies the following two conditions:
\begin{itemize}
\item[(i)] The empty arrangements are contained in $\mathcal{IF}$.
\item[(i\hspace{-0.5mm}i)] If there exists $H\in\mathscr{A}$ such that $\mathscr{A}\setminus\{H\}\in\mathcal{IF}$, $\mathscr{A}^H\in\mathcal{IF}$, and $\exp(\mathscr{A}^H)\subseteq\exp(\mathscr{A}\setminus\{H\})$, then $\mathscr{A}\in\mathcal{IF}$.
\end{itemize}
We note that every inductively free arrangement is free by Theorem \ref{thm-add-del-terao}.
Clearly all central $2$-arrangements are inductively free.
In addition, it is widely known that the braid arrangement $\mathcal{A}_{\ell-1}$ is inductively free.

\begin{Defi}
For a tuple $\bm{n}\in \mathbb{Z}_{\geq 0}^{V}$, we define $|\bm{n}|:=\sum_{i\in V}n_i$.
\end{Defi}
\begin{Them}\label{thm-free-A(n,G)-G-empty}
If $|V|=\ell$ and $E=\emptyset$, then $c\mathscr{A}(\bm{n},G)$ is inductively free with exponents $\{0,1,|\bm{n}|+1,|\bm{n}|+2,\dots,|\bm{n}|+\ell-1\}$.
\end{Them}
\begin{proof}
We can regard $V$ as $\{1,\dots,\ell\}$.
We prove the assertion by double induction on $\ell$ and $|\bm{n}|$.
If $\ell=2$ and $|\bm{n}|\geq 0$, then $c\mathscr{A}(\bm{n},G)$ is a product of the empty arrangement and a central $2$-arrangement with $n_1+n_2+1$ hyperplanes. Therefore $c\mathscr{A}(\bm{n},G)$ is inductively free with exponents $\{0,1,|\bm{n}|+1\}$.
If $|\bm{n}|=0$ and $\ell\geq 2$, then $\bm{n}=\bm{0}$.
Since $c\mathscr{A}(\bm{n},G)=\mathcal{A}_{\ell-1}\times\{\{z=0\}\}$, the arrangement $c\mathscr{A}(\bm{n},G)$ is inductively free with exponents $\{0,1,1,2,\dots,\ell-1\}$.

Let $\ell>2$ and $|\bm{n}|>0$.
Since there is a non-zero entry in $\bm{n}$, we may assume $n_{\ell}>0$.
We define $\bm{n}^{\prime}:=(n_1,\dots,n_{\ell-1},n_{\ell}-1)$ and $G_k:=(V,\{(\ell,i)\mid i=1,\dots,k\})$ for $k=0,1,\dots,\ell-1$.
For any $i=1,\dots,\ell-1$, since $\varepsilon_G(\ell,i)=0$ and $\varepsilon_{G_{\ell-1}}(\ell,i)=1$, we have $n_{\ell}+\varepsilon_{G}(\ell,i)=n_{\ell}-1+\varepsilon_{G_{\ell-1}}(\ell,i)$.
Therefore
\begin{align*}
\mathscr{A}(\bm{n},G)=\mathscr{A}(\bm{n}^{\prime},G_{\ell-1}).
\end{align*}
Here we prove that $c\mathscr{A}(\bm{n}^{\prime},G_k)$ is inductively free with exponents $\{0,1,|\bm{n}|+1,\dots,|\bm{n}|+\ell-2,|\bm{n}|+k\}$ by induction on $k$.
Let $k=0$. Since $G_0=(V,\emptyset)$ and $|\bm{n}^{\prime}|=|\bm{n}|-1$, the arrangement $c\mathscr{A}(\bm{n}^{\prime},G_0)$ is inductively free with exponents $\{0,1,|\bm{n}|,|\bm{n}|+1,\dots,|\bm{n}|+\ell-2\}$ by the induction hypothesis.

Let $k>0$. Let
\begin{align*}
H_k:=\{x_k-x_{\ell}=n_{\ell}\}=\{x_k-x_{\ell}=n_{\ell}-1+\varepsilon_{G_{k}}(\ell,k)\}\in\mathscr{A}(\bm{n}^{\prime},G_{k})
\end{align*}
and
\begin{align*}
\overline{H}_k:=\{x_k-x_{\ell}=n_{\ell}z\}\in c\mathscr{A}(\bm{n}^{\prime},G_{k}).
\end{align*}
Then the arrangement $c\mathscr{A}(\bm{n}^{\prime},G_k)\setminus\{\overline{H}_k\}=c\mathscr{A}(\bm{n}^{\prime},G_{k-1})$ is inductively free with exponents $\{0,1,|\bm{n}|+1,\dots,|\bm{n}|+\ell-2,|\bm{n}|+k-1\}$ by the induction hypothesis.
At the same time, by Theorem \ref{thm-restAnGtoH}, we have
\begin{align*}
\left(c\mathscr{A}(\bm{n}^{\prime},G_k)\right)^{\overline{H}_k}
=c\left(\mathscr{A}(\bm{n}^{\prime},G_k)^{H_k}\right)
=c\left(\mathscr{A}((\bm{n}^{\prime})^{H_k},G_k^{H_k})\right)
=c\mathscr{A}\left((\bm{n}^{\prime})^{H_k},(V^{H_k},\emptyset))\right)
\end{align*}
and $(\bm{n}^{\prime})^{H_k}=(n^{\prime}_i)_{i\in V^{H_k}}$, where $n^{\prime}_i=
\begin{cases}
n_i&\text{if}\ i\in V^{H_k}\setminus\{u\},\\
\max\{n_k+n_{\ell},n_{\ell}-1\}=n_k+n_{\ell}&\text{if}\ i=u.
\end{cases}
$
Since $\left|V^{H_k}\right|=\ell-1$ and $\left|\left(\bm{n}^{\prime}\right)^{H_i}\right|=|\bm{n}|$, the arrangement $\left(c\mathscr{A}(\bm{n}^{\prime},G_k)\right)^{\overline{H}_k}$ is inductively free with exponents $\{0,1,|\bm{n}|+1,\dots,|\bm{n}|+\ell-2\}$ by the induction hypothesis.
Hence the arrangement $c\mathscr{A}(\bm{n}^{\prime},G_k)$ is inductively free with exponents $\{0,1,|\bm{n}|+1,\dots,|\bm{n}|+\ell-2,|\bm{n}|+k\}$ by Theorem \ref{thm-add-del-terao}.
In particular, $c\mathscr{A}(\bm{n}^{\prime},G_{\ell-1})=c\mathscr{A}(\bm{n},G)$ is inductively free with exponents $\{0,1,|\bm{n}|+1,\dots,|\bm{n}|+\ell-2,|\bm{n}|+\ell-1\}$.
\end{proof}
\begin{Cor}\label{cor-ext-Cat-free}
The cone of the extended Catalan arrangement of type A is inductively free.
\end{Cor}

\section{Freeness for $\mathscr{A}(\bm{n},G)$}\label{sec-free-A(n,G)}
In this section, we generalize the freeness characterization by Athanasiadis for $\mathscr{A}_m(G)$ as follows.
\begin{Them}\label{thm-A(n,G)-free-A1A2}
Let $\bm{n}\in\mathbb{Z}_{\geq 0}^{V}$. Then $c\mathscr{A}(\bm{n},G)$ is free if and only if there exists a total order $\preceq$ on $V$ such that $G$ satisfies (A1) and (A2), while in this case, the multi-set of exponents of $c\mathscr{A}(\bm{n},G)$ is $\{0,1\}\cup\{|\bm{n}|+\ell-\ord(i)+1+b_i\mid i\in V\setminus\{\min_{\preceq}(V)\}\}$, where $\ord(i):=\#\{j\in V\mid j\preceq i\}$ and $b_i:=\#\{j\in V\mid j\prec i,(i,j)\in E\}+\#\{j\in V\mid j\prec i,(j,i)\in E\}$ for $i\in V\setminus\{\min_{\preceq}(V)\}$.
\end{Them}
To prove Theorem \ref{thm-A(n,G)-free-A1A2}, we introduce several results for characteristic polynomials, signed graphs, and multi-arrangements.
For an arrangement $\mathscr{A}$, let $\chi_{\mathscr{A}}(t)$ denote the characteristic polynomial of $\mathscr{A}$ (see \cite[Definition 2.52]{Orlik-Terao}).
\begin{Them}[cf. Corollary 2.57 in \cite{Orlik-Terao}]\label{thm-del-rest-formula}
Let $\mathscr{A}$ be an arrangement. For any $H\in\mathscr{A}$,
\begin{align*}
\chi_{\mathscr{A}}(t)=\chi_{\mathscr{A}\setminus\{H\}}(t)-\chi_{\mathscr{A}^H}(t).
\end{align*}
\end{Them}
\begin{Them}[Terao's factorization theorem \cite{Terao-fac}]\label{thm-terao-factr}
Let $\mathscr{A}$ be a central arrangement. If $\mathscr{A}$ is free with $\exp(\mathscr{A})=\{e_1,\dots,e_{\ell}\}$, then $\chi_{\mathscr{A}}(t)=\prod_{i=1}^{\ell}(t-e_i)$.
\end{Them}
Theorem \ref{thm-del-rest-formula} is a fundamental useful tool to calculate characteristic polynomials.
In the following example and lemma, we often use the contraposition of Theorem \ref{thm-terao-factr} to prove that certain arrangements are not free.

We first prove Theorem \ref{thm-A(n,G)-free-A1A2} in the case when $\ell=3$ by determining freeness for $c\mathscr{A}(\bm{n},G)$ for all digraphs $G$ with $3$ vertices.
We give two examples.
\begin{Examp}\label{ex-A(n,G)-del-rest}
(1)\ Let $V=\{i,s,t\}$ and $E=\{(t,s)\}$.
If $i\prec s\prec t$, then $G$ satisfies (A1) and (A2).
Let $H=\{x_s-x_t=n_t+\varepsilon_{G}(t,s)\}=\{x_s-x_t=n_t+1\}\in\mathscr{A}(\bm{n},G)$.
Then $\mathscr{A}(\bm{n},G)\setminus\{H\}=\mathscr{A}(\bm{n},(V,\emptyset))$.
In other words, the deletion of the hyperplane $H$ corresponds to the deletion of the edge $(t,s)$.
Next we consider the restriction of $\mathscr{A}(\bm{n},G)$ to $\{x_s-x_t=n_t+1\}$.
Since $n_s+n_t+1>n_t$, we have $n_u^{H}=n_s+n_t+1$, while clearly $E^H=\emptyset$.
Thus $\bm{n}^H=(n_i^H,n_u^H)=(n_i,n_s+n_t+1)$ and $\mathscr{A}(\bm{n},G)^{H}=\mathscr{A}(\bm{n}^H,(\{i,u\},\emptyset))$.
By Theorem \ref{thm-free-A(n,G)-G-empty}, $c(\mathscr{A}(\bm{n},G)\setminus\{H\})$ is free with exponents $\{0,1,|\bm{n}|+1,|\bm{n}|+2\}$, while $c\mathscr{A}\left(\bm{n},G\right)^H$ is free with exponents $\{0,1,|\bm{n}|+2\}$.
Therefore by Theorem \ref{thm-add-del-terao}, $c\mathscr{A}(\bm{n},G)$ is free with exponents $\{0,1,|\bm{n}|+2,|\bm{n}|+2\}$.
\begin{figure}[H]
 \centering
\caption{The pairs of integer tuples and digraphs corresponding to the arrangement $\mathscr{A}(\bm{n},G)$, the deletion $\mathscr{A}(\bm{n},G)\setminus\{H\}$ and the restriction $\mathscr{A}(\bm{n},G)^H$ in Example \ref{ex-A(n,G)-del-rest} (1)}

\begin{tabular}{c@{\hspace{15mm}}c@{\hspace{15mm}}c}
$\mathscr{A}(\bm{n},G)$&$\mathscr{A}(\bm{n},G)\setminus\{H\}$&$\mathscr{A}(\bm{n},G)^H$\\
$(n_i,n_s,n_t)$&$(n_i,n_s,n_t)$&$(n_i,n_s+n_t+1)$\\
\begin{tikzpicture}[node distance={15mm}, main/.style = {draw, circle}] 
\node[main] (1) {$s$}; \node[main] (2) [below left of=1] {$t$}; \node[main] (3) [below right of=1] {$i$};
\draw[->,>=stealth] (2) -- (1);
\end{tikzpicture}
&\begin{tikzpicture}[node distance={15mm}, main/.style = {draw, circle}] 
\node[main] (1) {$s$}; \node[main] (2) [below left of=1] {$t$}; \node[main] (3) [below right of=1] {$i$};
\end{tikzpicture}
&\begin{tikzpicture}[node distance={15mm}, main/.style = {draw, circle}] 
\node[main] (1) {$u$}; \node[main] (2) [right of=1] {$i$};
\end{tikzpicture} 
\end{tabular}
\end{figure}
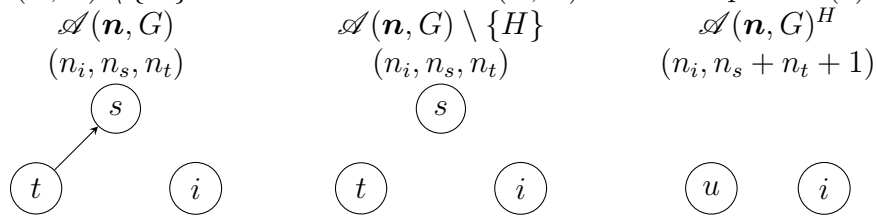

(2)\ Let $V=\{i,s,t\}$, and let $E=\{(i,t),(t,s)\}$.
Then there are no total orders on $V$ such that $G$ satisfies (A1) and (A2).
Let $H=\{x_s-x_t=n_t+1\}\in\mathscr{A}(\bm{n},G)$.
Similarly to (1), we have that $\mathscr{A}(\bm{n},G)\setminus\{H\}=\mathscr{A}(\bm{n},(V,\{(i,t)\}))$ and $\mathscr{A}\left(\bm{n},G\right)^{H}=\mathscr{A}(\bm{n}^H,(\{i,u\},\{(i,u)\}))$, where $\bm{n}^H=(n_i^H,n_u^H)=(n_i,n_s+n_t+1)$. By Theorem \ref{thm-del-rest-formula} and Theorem \ref{thm-terao-factr}, we have $\chi_{c\mathscr{A}(\bm{n},G)}(t)=t(t-1)\left(t^2-(2|\bm{n}|+5)t+|\bm{n}|^2+5|\bm{n}|+7\right)$, and hence $c\mathscr{A}(\bm{n},G)$ is not free.
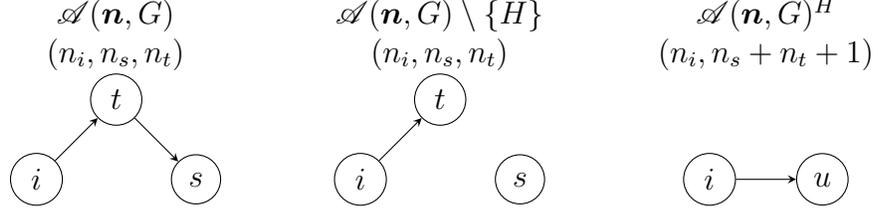
\begin{figure}[H]
 \centering
\caption{The pairs of integer tuples and digraphs corresponding to the arrangement $\mathscr{A}(\bm{n},G)$, the deletion $\mathscr{A}(\bm{n},G)\setminus\{H\}$, and the restriction $\mathscr{A}(\bm{n},G)^H$ in Example \ref{ex-A(n,G)-del-rest} (2)}

\begin{tabular}{c@{\hspace{15mm}}c@{\hspace{15mm}}c}
$\mathscr{A}(\bm{n},G)$&$\mathscr{A}(\bm{n},G)\setminus\{H\}$&$\mathscr{A}(\bm{n},G)^H$\\
$(n_i,n_s,n_t)$&$(n_i,n_s,n_t)$&$(n_i,n_s+n_t+1)$\\
\begin{tikzpicture}[node distance={15mm}, main/.style = {draw, circle}] 
\node[main] (1) {$t$}; \node[main] (2) [below left of=1] {$i$}; \node[main] (3) [below right of=1] {$s$};
\draw[->,>=stealth] (2) -- (1); \draw[->,>=stealth] (1) -- (3);
\end{tikzpicture}
&\begin{tikzpicture}[node distance={15mm}, main/.style = {draw, circle}] 
\node[main] (1) {$t$}; \node[main] (2) [below left of=1] {$i$}; \node[main] (3) [below right of=1] {$s$};
\draw[->,>=stealth] (2) -- (1);
\end{tikzpicture}
&\begin{tikzpicture}[node distance={15mm}, main/.style = {draw, circle}] 
\node[main] (1) {$i$}; \node[main] (2) [right of=1] {$u$}; \draw[->,>=stealth] (1) -- (2);
\end{tikzpicture} 
\end{tabular}
\end{figure}
\end{Examp}
We define digraphs $G_3:=(\{i,s,t\},\{(i,t),(t,s)\})$,
$G_{10}:=(\{i,s,t\},\{(i,t),(t,s),(s,i)\})$, and $G_{13}:=(\{i,s,t\},\{(i,t),(t,s),(s,i),(i,s)\})$. (We use the same notation as in Table \ref{table-AnG-ord3-1} to Table \ref{table-AnG-ord3-4}.)
\begin{Lem}\label{lem-A(n,G)-free-A1A2-l=3}
Let $\ell=3$. Then the following are equivalent.
\begin{itemize}
\item[(1)]\ The arrangement $c\mathscr{A}(\bm{n},G)$ is free.
\item[(2)]\ There exists a total order $\preceq$ on $V$ such that $G$ satisfies (A1) and (A2).
\item[(3)]\ The graph $G$ is none of $G_3$, $G_{10}$ and $G_{13}$.
\end{itemize}
\end{Lem}
\begin{proof}
If $E$ is the empty set, then $c\mathscr{A}(\bm{n},G)$ is free with exponents $\{0,1,|\bm{n}|+1,|\bm{n}|+2\}$ by Theorem \ref{thm-free-A(n,G)-G-empty}.
By a straightforward computation using Theorem \ref{thm-add-del-terao}, Theorem \ref{thm-del-rest-formula}, and Theorem \ref{thm-terao-factr}, we can determine the freeness for $c\mathscr{A}(\bm{n},G)$ for the remaining digraphs $G$ with $3$ vertices. In addition, we can also determine whether there exists a total order $\preceq$ on $V$ such that $G$ satisfies (A1) and (A2) for all digraphs $G$ with $3$ vertices. See Table \ref{table-AnG-ord3-1} to Table \ref{table-AnG-ord3-4}. We also have that
\begin{align*}
\chi_{c\mathscr{A}(\bm{n},G_3)}(t)&=t(t-1)\left(t^2-(2|\bm{n}|+5)t+|\bm{n}|^2+5|\bm{n}|+7\right),\\
\chi_{c\mathscr{A}(\bm{n},G_{10})}(t)&=t(t-1)\left(t^2-(2|\bm{n}|+6)t+|\bm{n}|^2+6|\bm{n}|+11\right),\\
\chi_{c\mathscr{A}(\bm{n},G_{13})}(t)&=t(t-1)\left(t^2-(2|\bm{n}|+7)t+|\bm{n}|^2+7|\bm{n}|+13\right).
\end{align*}
\begin{table}[H]
\centering
\caption{Freeness for $c\mathscr{A}(\bm{n},G)$ and existence $\preceq$ which satisfies (A1) and (A2) for all digraphs with $3$ vertices (1)}\label{table-AnG-ord3-1}
\scalebox{0.88}{
\begin{tabular}{|c|cccc|}
\hline
 & $G_1$ & $G_2$ & $G_3$ & $G_4$ \\
Digraph&
\begin{tikzpicture}[node distance={15mm}, main/.style = {draw, circle}] 
\node[main] (1) {}; \node[main] (2) [below left of=1] {}; \node[main] (3) [below right of=1] {};
\end{tikzpicture}
&\begin{tikzpicture}[node distance={15mm}, main/.style = {draw, circle}] 
\node[main] (1) {}; \node[main] (2) [below left of=1] {}; \node[main] (3) [below right of=1] {};
\draw[->,>=stealth] (2) -- (1);
\end{tikzpicture}
&\begin{tikzpicture}[node distance={15mm}, main/.style = {draw, circle}] 
\node[main] (1) {}; \node[main] (2) [below left of=1] {}; \node[main] (3) [below right of=1] {};
\draw[->,>=stealth] (2) -- (1); \draw[->,>=stealth] (1) -- (3);
\end{tikzpicture}
&\begin{tikzpicture}[node distance={15mm}, main/.style = {draw, circle}] 
\node[main] (1) {}; \node[main] (2) [below left of=1] {}; \node[main] (3) [below right of=1] {};
\draw[->,>=stealth] (2) -- (1); \draw[->,>=stealth] (3) -- (1);
\end{tikzpicture}\\
\hline
$|E|$&$0$&$1$&$2$&$2$\\
(A1),\ (A2)&Yes&Yes&No&Yes\\
Free&Yes&Yes&No&Yes\\
exponents&$\{0,1,|\bm{n}|+1,|\bm{n}|+2\}$&$\{0,1,|\bm{n}|+2,|\bm{n}|+2\}$&N/A&$\{0,1,|\bm{n}|+2,|\bm{n}|+3\}$\\
\hline
\end{tabular}
}
\end{table}
\begin{table}[H]
\centering
\caption{Freeness for $c\mathscr{A}(\bm{n},G)$ and existence $\preceq$ which satisfies (A1) and (A2) for all digraphs with $3$ vertices (2)}\label{table-AnG-ord3-2}
\scalebox{0.88}{
\begin{tabular}{|c|cccc|}
\hline
& $G_5$ & $G_6$ & $G_7$ & $G_8$ \\
Digraph&
\begin{tikzpicture}[node distance={15mm}, main/.style = {draw, circle}] 
\node[main] (1) {}; \node[main] (2) [below left of=1] {}; \node[main] (3) [below right of=1] {};
\draw[->,>=stealth] (2) -- (1); \draw[->,>=stealth] (2) -- (3);
\end{tikzpicture}
&\begin{tikzpicture}[node distance={15mm}, main/.style = {draw, circle}] 
\node[main] (1) {}; \node[main] (2) [below left of=1] {}; \node[main] (3) [below right of=1] {};
\draw[->,>=stealth] (3) to [out=200,in=340,looseness=0.7] (2); \draw[->,>=stealth] (2) to [out=20,in=160,looseness=0.7] (3);
\end{tikzpicture}
&\begin{tikzpicture}[node distance={15mm}, main/.style = {draw, circle}] 
\node[main] (1) {}; \node[main] (2) [below left of=1] {}; \node[main] (3) [below right of=1] {};
\draw[->,>=stealth] (2) -- (1); \draw[->,>=stealth] (3) to [out=200,in=340,looseness=0.7] (2); \draw[->,>=stealth] (2) to [out=20,in=160,looseness=0.7] (3);
\end{tikzpicture}
&\begin{tikzpicture}[node distance={15mm}, main/.style = {draw, circle}] 
\node[main] (1) {}; \node[main] (2) [below left of=1] {}; \node[main] (3) [below right of=1] {};
\draw[->,>=stealth] (1) -- (2); \draw[->,>=stealth] (3) to [out=200,in=340,looseness=0.7] (2); \draw[->,>=stealth] (2) to [out=20,in=160,looseness=0.7] (3);
\end{tikzpicture}\\
\hline
$|E|$&$2$&$2$&$3$&$3$\\
(A1),\ (A2)&Yes&Yes&Yes&Yes\\
Free&Yes&Yes&Yes&Yes\\
exponents&$\{0,1,|\bm{n}|+2,|\bm{n}|+3\}$&$\{0,1,|\bm{n}|+2,|\bm{n}|+3\}$&$\{0,1,|\bm{n}|+3,|\bm{n}|+3\}$&$\{0,1,|\bm{n}|+3,|\bm{n}|+3\}$\\
\hline
\end{tabular}
}
\end{table}
\begin{table}[H]
\centering
\caption{Freeness for $c\mathscr{A}(\bm{n},G)$ and existence $\preceq$ which satisfies (A1) and (A2) for all digraphs with $3$ vertices (3)}\label{table-AnG-ord3-3}
\scalebox{0.88}{
\begin{tabular}{|c|cccc|}
\hline
& $G_9$ & $G_{10}$ & $G_{11}$ & $G_{12}$ \\
Digraph&
\begin{tikzpicture}[node distance={15mm}, main/.style = {draw, circle}] 
\node[main] (1) {}; \node[main] (2) [below left of=1] {}; \node[main] (3) [below right of=1] {};
\draw[->,>=stealth] (2) -- (1); \draw[->,>=stealth] (1) -- (3); \draw[->,>=stealth] (2) -- (3);
\end{tikzpicture}
&\begin{tikzpicture}[node distance={15mm}, main/.style = {draw, circle}] 
\node[main] (1) {}; \node[main] (2) [below left of=1] {}; \node[main] (3) [below right of=1] {};
\draw[->,>=stealth] (2) -- (1); \draw[->,>=stealth] (1) -- (3); \draw[->,>=stealth] (3) -- (2);
\end{tikzpicture}
&\begin{tikzpicture}[node distance={15mm}, main/.style = {draw, circle}] 
\node[main] (1) {}; \node[main] (2) [below left of=1] {}; \node[main] (3) [below right of=1] {};
\draw[->,>=stealth] (1) -- (2); \draw[->,>=stealth] (1) -- (3); \draw[->,>=stealth] (3) to [out=200,in=340,looseness=0.7] (2); \draw[->,>=stealth] (2) to [out=20,in=160,looseness=0.7] (3);
\end{tikzpicture}
&\begin{tikzpicture}[node distance={15mm}, main/.style = {draw, circle}] 
\node[main] (1) {}; \node[main] (2) [below left of=1] {}; \node[main] (3) [below right of=1] {};
\draw[->,>=stealth] (2) -- (1); \draw[->,>=stealth] (3) -- (1); \draw[->,>=stealth] (3) to [out=200,in=340,looseness=0.7] (2); \draw[->,>=stealth] (2) to [out=20,in=160,looseness=0.7] (3);
\end{tikzpicture}\\
\hline
$|E|$&$3$&$3$&$4$&$4$\\
(A1),\ (A2)&Yes&No&Yes&Yes\\
Free&Yes&No&Yes&Yes\\
exponents&$\{0,1,|\bm{n}|+3,|\bm{n}|+3\}$&N/A&$\{0,1,|\bm{n}|+3,|\bm{n}|+4\}$&$\{0,1,|\bm{n}|+3,|\bm{n}|+4\}$\\
\hline
\end{tabular}
}
\end{table}
\begin{table}[H]
\centering
\caption{Freeness for $c\mathscr{A}(\bm{n},G)$ and existence $\preceq$ which satisfies (A1) and (A2) for all digraphs with $3$ vertices (4)}\label{table-AnG-ord3-4}
\scalebox{0.88}{
\begin{tabular}{|c|cccc|}
\hline
& $G_{13}$ & $G_{14}$ & $G_{15}$ & $G_{16}$ \\
Digraph&
\begin{tikzpicture}[node distance={15mm}, main/.style = {draw, circle}] 
\node[main] (1) {}; \node[main] (2) [below left of=1] {}; \node[main] (3) [below right of=1] {};
\draw[->,>=stealth] (2) -- (1); \draw[->,>=stealth] (1) -- (3); \draw[->,>=stealth] (3) to [out=200,in=340,looseness=0.7] (2); \draw[->,>=stealth] (2) to [out=20,in=160,looseness=0.7] (3);
\end{tikzpicture}
&\begin{tikzpicture}[node distance={15mm}, main/.style = {draw, circle}] 
\node[main] (1) {}; \node[main] (2) [below left of=1] {}; \node[main] (3) [below right of=1] {};
\draw[->,>=stealth] (2) to [out=65,in=200,looseness=0.7] (1); \draw[->,>=stealth] (1) to [out=245,in=25,looseness=0.7] (2); \draw[->,>=stealth] (3) to [out=200,in=340,looseness=0.7] (2); \draw[->,>=stealth] (2) to [out=20,in=160,looseness=0.7] (3);
\end{tikzpicture}
&\begin{tikzpicture}[node distance={15mm}, main/.style = {draw, circle}] 
\node[main] (1) {}; \node[main] (2) [below left of=1] {}; \node[main] (3) [below right of=1] {};
\draw[->,>=stealth] (2) to [out=65,in=200,looseness=0.7] (1); \draw[->,>=stealth] (1) to [out=245,in=25,looseness=0.7] (2); \draw[->,>=stealth] (3) to [out=200,in=340,looseness=0.7] (2); \draw[->,>=stealth] (2) to [out=20,in=160,looseness=0.7] (3); \draw[->,>=stealth] (1) -- (3);
\end{tikzpicture}
&\begin{tikzpicture}[node distance={15mm}, main/.style = {draw, circle}] 
\node[main] (1) {}; \node[main] (2) [below left of=1] {}; \node[main] (3) [below right of=1] {};
\draw[->,>=stealth] (2) to [out=65,in=200,looseness=0.7] (1); \draw[->,>=stealth] (1) to [out=245,in=25,looseness=0.7] (2); \draw[->,>=stealth] (3) to [out=200,in=340,looseness=0.7] (2); \draw[->,>=stealth] (2) to [out=20,in=160,looseness=0.7] (3); \draw[->,>=stealth] (1) to [out=335,in=115,looseness=0.7] (3); \draw[->,>=stealth] (3) to [out=155,in=295,looseness=0.7] (1);
\end{tikzpicture}\\
\hline
$|E|$&$4$&$4$&$5$&$6$\\
(A1),\ (A2)&No&Yes&Yes&Yes\\
Free&No&Yes&Yes&Yes\\
exponents&N/A&$\{0,1,|\bm{n}|+3,|\bm{n}|+4\}$&$\{0,1,|\bm{n}|+4,|\bm{n}|+4\}$&$\{0,1,|\bm{n}|+4,|\bm{n}|+5\}$\\
\hline
\end{tabular}
}
\end{table}
\end{proof}

A signed graph is an undirected graph in which the set of edges is the union of sets $E_{+}$ and $E_{-}$ of edges with $E_{+}\cap E_{-}=\emptyset$ (for details, see \cite{Abe-Nuida-Numata,Nuida}).
In this paper, we only use a signed graph obtained by a digraph.
For a digraph $G=(V,E)$, we define a signed graph $S(G):=(V,E_{+}\cup E_{-})$ as follows.
\begin{itemize}
\item If $(i,j)\in E$ and $(j,i)\in E$, then $\{i,j\}\in E_{+}$.
\item If either $(i,j)$ or $(j,i)$ is contained in $E$,
then $\{i,j\}\not\in E_{+}\cup E_{-}$.
\item If $(i,j)\not\in E$ and $(j,i)\not\in E$, then $\{i,j\}\in E_{-}$.
\end{itemize}
If there exists a total order $\preceq$ on $V$ such that
$G$ satisfies (A1) and (A2), then $S(G)$ is signed eliminable
(introduced in \cite{Abe-Nuida-Numata}), that is, there exists a total order $\preceq$ on $V$ such that for $i,j,k\in V$,
\begin{itemize}
\item if $i\prec k$, $j\prec k$, $\{k,i\}\in E_{\mu}$ and $\{i,j\}\in E_{\nu}$
for $\{\mu,\nu\}=\{+,-\}$, then $\{k,j\}\in E_{\nu}$,
\item if $i\prec k$, $j\prec k$, $\{k,i\}\in E_{\mu}$ and $\{k,j\}\in E_{\mu}$
for $\mu\in\{+,-\}$, then $\{i,j\}\in E_{\mu}$.
\end{itemize}
Moreover, Abe \cite{Abe} characterized digraphs
satisfying (A1) and (A2) by using signed eliminable graphs.
\begin{Them}[Proposition 2.1 in \cite{Abe}]\label{thm-A1A2-SE}
Let $G=(V,E)$ be a digraph.
There exists a total order $\preceq$ on $V$ such that $G$ satisfies (A1) and (A2) if and only if $S(G)$ is signed eliminable and $G$ does not contain $G_3$, $G_{10}$ and $G_{13}$ as induced subgraphs.
\end{Them}

We next introduce important results for multi-arrangements.
A multi-arrangement is a pair $(\mathscr{A},m)$ of a central $\ell$-arrangement
$\mathscr{A}$ and a map $m:\mathscr{A}\rightarrow \mathbb{Z}_{\geq 0}$.
We say $(\mathscr{A},m)$ is free if
\begin{align*}
D(\mathscr{A},m):=\left\{\theta\in\sum_{i\in V}S\partial_i\middle|\theta(\alpha_H)\in\alpha_H^{m(H)} S\ {\rm for\ all}\ H\in\mathscr{A}\right\}
\end{align*}
is a free $S$-module, while we define a multi-set of exponents by $\{\deg\theta_1,\dots,\deg\theta_{\ell}\}$, where $\{\theta_1,\dots,\theta_{\ell}\}$ is a homogeneous basis for $D(\mathscr{A},m)$.
Let $H_0\in\mathscr{A}$. The Ziegler multiplicity $m_{H_0}$ is defined by
$m_{H_0}(X):=\#\{H^\prime\in\mathscr{A}\setminus\{H_0\}\mid H^\prime\cap H_0=X\}$ for $X\in\mathscr{A}^{H_0}$, while $(\mathscr{A}^{H_0},m_{H_0})$ is called the Ziegler restriction with respect to $H_0$.
\begin{Them}[Theorem 11 in \cite{Ziegler-multiarr}]\label{thm-Zie-rest}
If a central $\ell$-arrangement $\mathscr{A}$ is free with exponents $\{1,e_2,\dots,e_{\ell}\}$, then $(\mathscr{A}^{H_0},m_{H_0})$ is free with exponents $\{e_2,\dots,e_{\ell}\}$.
\end{Them}
A localization is defined by $\mathscr{A}_X:=\{H\in\mathscr{A}\mid X\subseteq H\}$ for $X\in L(\mathscr{A})$.
\begin{Them}[Theorem 2.2 in \cite{Yoshinaga-ER}]\label{thm-Yoshi-criterion}
Let $\mathscr{A}$ be a central $\ell$-arrangement with $\ell\geq 4$, and let $H_0\in\mathscr{A}$.
Then $\mathscr{A}$ is free if and only if $(\mathscr{A}^{H_0},m_{H_0})$
is free and the localization $\mathscr{A}_X$ is free for any
$X\in L\left(\mathscr{A}^{H_0}\right)\setminus
\left\{\bigcap_{H\in\mathscr{A}}H\right\}$.
\end{Them}
Let $\overline{G}=(V,E_{+}\cup E_{-})$ be a signed graph.
A map $m_{\overline{G}}:\mathcal{A}_{\ell-1}\rightarrow \{-1,0,1\}$ is defined by
\begin{align*}
m_{\overline{G}}\left(\{x_i-x_j=0\}\right)=
\begin{cases}
1&\quad{\rm if}\ \{i,j\}\in E_{+},\\
-1&\quad{\rm if}\ \{i,j\}\in E_{-},\\
0&\quad{\rm otherwise},
\end{cases}
\end{align*}
while for $k\in\mathbb{Z}_{\geq 0}$ and $\bm{n}=(n_i)_{i\in V}\in\mathbb{Z}_{\geq 0}^{V}$, $\mathcal{A}_{\ell-1}(k,\bm{n})[\overline{G}]$ is defined by the multi-arrangement $(\mathcal{A}_{\ell-1},m)$, where $m\left(\{x_i-x_j=0\}\right)=2k+n_i+n_j+m_{\overline{G}}(\{x_i-x_j=0\})$.
We note that if $G$ is a digraph, then
\begin{align}\label{eq-m_SG-varepsilon}
m_{S(G)}(\{x_i-x_j=0\})=\varepsilon_{G}(i,j)+\varepsilon_{G}(j,i)-1.
\end{align}
The following is given by Abe, Nuida and Numata \cite{Abe-Nuida-Numata}.
\begin{Them}[Theorem 0.3 in \cite{Abe-Nuida-Numata}]\label{thm-A_ellka-free}
Let $\overline{G}=(V,E_{+}\cup E_{-})$ be a signed graph.
We assume that one of the following three is satisfied;
\begin{itemize}
\item[(i)]\ $k>0$,
\item[(i\hspace{-0.5mm}i)]\ $E_{-}=\emptyset$, or
\item[(i\hspace{-0.5mm}i\hspace{-0.5mm}i)]\ $k=0$, $m\left(\{x_i-x_j=0\}\right)>0$ for all $i,j\in V$ with $i\neq j$, and all the triples $\{s,i,j\}$ with $\{i,j\}\in E_{-}$ and $m(\{x_s-x_i=0\})<m(\{x_s-x_j=0\})$, it holds that $n_i>0$.
\end{itemize}
Then the multi-arrangement $\mathcal{A}_{\ell-1}(k,\bm{n})[\overline{G}]$ is free if and only if $\overline{G}$ is signed eliminable.
In this case, the exponents are
\begin{align*}
\{0\}\cup\left\{|\bm{n}|+k\ell+d_i\,\middle|\, i\in V\setminus\left\{\min_{\preceq}(V)\right\}\right\},
\end{align*}
where
\begin{align*}
d_i:=\#\{j\in V\mid j\prec i,\{j,i\}\in E_{+}\}-\#\{j\in V\mid j\prec i,\{j,i\}\in E_{-}\}
\end{align*}
for $i\in V\setminus\{\min_{\preceq}(V)\}$.
\end{Them}
We denote by $\Phi_r$ the empty $r$-arrangement.
According to the preparations above, we can prove Theorem \ref{thm-A(n,G)-free-A1A2}.
The proof is essentially the same as that of \cite[Theorem 5.3]{Abe-Nuida-Numata} and \cite[Theorem 1.2]{Abe}.
\begin{proof}[Proof of Theorem \ref{thm-A(n,G)-free-A1A2}]
When $\ell=2$, the assertion is obvious.
We prove by induction on $\ell\geq 3$ that if there exists a total order $\preceq$ on $V$ such that $G$ satisfies (A1) and (A2),
then $c\mathscr{A}(\bm{n},G)$ is free.
When $\ell=3$, it is already proved in Lemma \ref{lem-A(n,G)-free-A1A2-l=3}.
Let $\ell>3$.
We use Theorem \ref{thm-Yoshi-criterion} to prove the assertion.
Let $H_0:=\{z=0\}\in c\mathscr{A}(\bm{n},G)$.
Then $c\mathscr{A}(\bm{n},G)^{H_0}=\mathcal{A}_{\ell-1}$ and
\begin{align*}
m_{H_0}(\{x_i-x_j=0\})=n_i+n_j+\varepsilon_{G}(i,j)+\varepsilon_{G}(j,i)+1
=2+n_i+n_j+m_{S(G)}(\{x_i-x_j=0\})
\end{align*}
by the equality \eqref{eq-m_SG-varepsilon}.
Therefore $\left(c\mathscr{A}(\bm{n},G)^{H_0},m_{H_0}\right)
=\mathcal{A}_{\ell-1}(1,\bm{n})[S(G)]$.
By Theorem \ref{thm-A1A2-SE}, $S(G)$ is signed eliminable, and hence by Theorem \ref{thm-A_ellka-free}, $\left(c\mathscr{A}(\bm{n},G)^{H_0},m_{H_0}\right)$ is free.
Let $X\in L\left(c\mathscr{A}(\bm{n},G)\right)$ with $\{x_{i_1}=\cdots=x_{i_{\ell}},z=0\}\subsetneq X\subseteq H_0$, where $V=\{i_1,\dots,i_{\ell}\}$.
Then we can write $X=\{x_{i^{\prime}_1}=\cdots=x_{i^{\prime}_r},z=0\}$ for some $i^{\prime}_1,\dots,i^{\prime}_r\in V$.
Let $G^{\prime}$ be the induced subgraph of $G$ with vertices $i^{\prime}_1,\dots,i^{\prime}_r$.
Since $G$ satisfies (A1) and (A2), $G^{\prime}$ also satisfies (A1) and (A2).
By the induction hypothesis, $c\mathscr{A}((n_{i^{\prime}_1},\dots,n_{i^{\prime}_r}),G^{\prime})$ is free.
Therefore $\left(c\mathscr{A}(\bm{n},G)\right)_{X}=c\mathscr{A}((n_{i^{\prime}_1},\dots,n_{i^{\prime}_r}),G^{\prime})\times\Phi_{\ell-r}$ is also free.
By Theorem \ref{thm-Yoshi-criterion}, we complete the ``if'' part of the assertion.

We next prove the converse. We assume that $c\mathscr{A}(\bm{n},G)$ is free.
By Theorem \ref{thm-Zie-rest}, the Ziegler restriction $\left(c\mathscr{A}(\bm{n},G)^{H_0},m_{H_0}\right) =\mathcal{A}_{\ell-1}(1,\bm{n})[S(G)]$
with respect to $H_0=\{z=0\}$ is free.
Then $S(G)$ is signed eliminable by Theorem \ref{thm-A_ellka-free}.
We recall the fact that if a central arrangement $\mathscr{A}$ is free,
then a localization $\mathscr{A}_X$ is free for any $X\in L(\mathscr{A})$
(see \cite[Theorem 4.37]{Orlik-Terao}).
If $G$ contains $G^{\prime}\in\{G_3,G_{10},G_{13}\}$ as a induced subgraph
with vertices $i_1,i_2,i_3$, then the localization
$c\mathscr{A}(\bm{n},G)_{\{x_{i_1}=x_{i_2}=x_{i_3},z=0\}}
=c\mathscr{A}((n_{i_1},n_{i_2},n_{i_3}),G^{\prime})\times\Phi_{\ell-3}$
is not free by Lemma \ref{lem-A(n,G)-free-A1A2-l=3}.
Therefore $S(G)$ is signed eliminable and $G$ does not contain
$G_3$, $G_{10}$ and $G_{13}$ as a induced subgraph.
By Theorem \ref{thm-A1A2-SE}, there exists a total order $\preceq$ on $V$ such that $G$ satisfies (A1) and (A2).
We complete the ``only if'' part of the assertion.

Finally we compute exponents of $c\mathscr{A}(\bm{n},G)$.
We suppose that $c\mathscr{A}(\bm{n},G)$ is free.
Then $\left(c\mathscr{A}(\bm{n},G)^{H_0},m_{H_0}\right)$ is free with exponents $\{0\}\cup\{|\bm{n}|+\ell+d_i\mid i\in V\setminus\{\min_{\preceq}(V)\}\}$ by Theorem \ref{thm-A_ellka-free}.
For $i\in V\setminus\{\min_{\preceq}(V)\}$, we have
\begin{align*}
b_i&=\#\{j\in V\mid j\prec i,(i,j)\in E\}+\#\{j\in V\mid j\prec i,(j,i)\in E\}\\
&=\#\{j\in V\mid j\prec i\}-\#\{j\in V\mid j\prec i,(i,j)\not\in E,(j,i)\not\in E\}\\
&\qquad\qquad\qquad\qquad\qquad\qquad+\#\{j\in V\mid j\prec i,(i,j)\in E,(j,i)\in E\}\\
&=\ord(i)-1-\#\{j\in V\mid j\prec i,\{j,i\}\in E_{-}\}+\#\{j\in V\mid j\prec i,\{j,i\}\in E_{+}\}\\
&=\ord(i)-1+d_i.
\end{align*}
Therefore by Theorem \ref{thm-Zie-rest}, the multi-set of exponents of $c\mathscr{A}(\bm{n},G)$ is 
\begin{align*}
\{0,1\}\cup\left\{|\bm{n}|+\ell-\ord(i)+1+b_i\,\middle|\, i\in V\setminus\left\{\min_{\preceq}(V)\right\}\right\}.
\end{align*}
\end{proof}

\noindent
{\it Remark.}\ 
When $V=\{1,\dots\ell\}$ and $G$ satisfies (A1) and (A2), we can prove by the finite field method that
\begin{align}\label{eq-athana-chiAnG}
\chi_{\mathscr{A}(\bm{n},G)}(t)=t\prod_{1<i\leq j\leq\ell}
\left(t-|\bm{n}|-\ell+i-1-b_i\right)
\end{align}
as a simple generalization of \cite[Theorem 3.9]{Athanasiadis1}.
On the other hand, we also have the equality \eqref{eq-athana-chiAnG} by Theorem \ref{thm-A(n,G)-free-A1A2} and Theorem \ref{thm-terao-factr}.
\begin{Cor}
Let $V=\{1,\dots,\ell\}$.
If $E=\{(j,i)\mid 1\leq i<j\leq\ell\}$, then $c\mathscr{A}(\bm{n},G)$ is free.
In particular, the cone of the extended Shi arrangement of type A is free.
\end{Cor}

\section{The extended Shi and Catalan arrangements}\label{sec-hereditarily-free}
A central arrangement $\mathscr{A}$ is said to be hereditarily inductively free if $\mathscr{A}^X$ is inductively free for any $X\in L(\mathscr{A})$.
If $\mathscr{A}$ is hereditarily inductively free, then $\mathscr{A}$ is hereditarily free.
Since
\begin{align*}
L(\mathscr{A})=\left\{\mathbb{K}^V\right\}\cup\bigcup_{H\in\mathscr{A}}\left\{H\cap X\,\middle|\, X\in L(\mathscr{A}\setminus\{H\})\right\}
=\left\{\mathbb{K}^V\right\}\cup\bigcup_{H\in\mathscr{A}}L\left(\mathscr{A}^H\right),
\end{align*}
for a central arrangement $\mathscr{A}$, we have that
\begin{enumerate}
\item[] $\mathscr{A}$ is hereditarily inductively free,
\item[$\Leftrightarrow$] $\mathscr{A}$ is inductively free, and $\left(\mathscr{A}^H\right)^{H\cap X}=\mathscr{A}^{H\cap X}$ is inductively free for any $H\in\mathscr{A}$ and $X\in L(\mathscr{A}\setminus\{H\})$,
\item[$\Leftrightarrow$] $\mathscr{A}$ is inductively free, and $\mathscr{A}^{H}$ is hereditarily inductively free for any $H\in\mathscr{A}$.
\end{enumerate}

Let $\overline{H}\in c\mathscr{A}(\bm{n},G)$.
If $\overline{H}=\{z=0\}$, then $\left(c\mathscr{A}(\bm{n},G)\right)^{\overline{H}}$ equals the inductively free arrangement $\mathcal{A}_{\ell-1}$.
If $\overline{H}\neq\{z=0\}$, then we can write $\overline{H}=\{x_{i}-x_{j}=cz\}$ for some $i,j\in V$ and $-n_i-\varepsilon_G(i,j)\leq c\leq n_j+\varepsilon_G(j,i)$, while $\left(c\mathscr{A}(\bm{n},G)\right)^{\overline{H}}=c\mathscr{A}\left(\bm{n},G\right)^H$, where $H=\{x_{i}-x_{j}=c\}$.
Therefore $c\mathscr{A}(\bm{n},G)$ is hereditarily inductively free if and only if $c\mathscr{A}(\bm{n},G)$ is inductively free and $c\mathscr{A}(\bm{n},G)^H$ is hereditarily inductively free for any $H\in\mathscr{A}(\bm{n},G)$.
The following is the main result of this paper.
\begin{Them}\label{thm-shi-catalan-hered-free}
Let $\bm{n}\in\mathbb{Z}_{\geq 0}^{V}$ and let $G=(V,E)$ be a digraph.
\begin{itemize}
\item[(1)] If $E=\emptyset$, then $c\mathscr{A}(\bm{n},G)$ is hereditarily inductively free.
\item[(2)] We assume that $V=\{1,\dots,\ell\}$ and $E=\{(j,i)\mid i,j\in V,\ i<j\}$. Then $c\mathscr{A}(\bm{n},G)$ is hereditarily free if and only if $\ell\leq 5$.
\end{itemize}
\end{Them}
\begin{proof}
(1)\ Let $E=\emptyset$. We prove the assertion by induction on $|V|$.
If $|V|=2$, clearly $c\mathscr{A}(\bm{n},G)$ is hereditarily inductively free.
Let $|V|>2$. Then $c\mathscr{A}(\bm{n},G)$ is inductively free by Theorem \ref{thm-free-A(n,G)-G-empty}.
Since $E^H=\emptyset$ for any $H\in \mathscr{A}(\bm{n},G)$, $c\mathscr{A}(\bm{n},G)^H=c\mathscr{A}(\bm{n}^H,G^H)$ is hereditarily inductively free by induction hypothesis.
Therefore $c\mathscr{A}(\bm{n},G)$ is hereditarily inductively free.

(2)\ Let $\ell\leq 5$. For any digraph $G^{\prime}$ obtained from contraction of $G$, we can give a total order such that $G^{\prime}$ satisfies the conditions (A1) and (A2) by computer\footnote{The authors verified it by using SageMath \cite{SageMath} computer programming.}.
This means that $c\mathscr{A}(\bm{n},G)$ is hereditarily free by Theorem \ref{thm-A(n,G)-free-A1A2}.

Conversely we prove that if $\ell\geq 6$, then $c\mathscr{A}(\bm{n},G)$ is not hereditarily free.
We first consider the case $\ell>6$.
Let $X=\left\{x_6=x_7=\cdots=x_{\ell}\right\}=\bigcap_{j=7}^{\ell}\left\{x_{j-1}-x_j=0\right\}\in L\left(c\mathscr{A}(\bm{n},G)\right)$, and we define
\begin{align*}
\bm{n}^{\prime}:=\left(n_1,n_2,n_3,n_4,n_5,\max\{n_6,n_7,\dots,n_{\ell}\}\right)\ \text{and}\ G^{\prime}:=\left(\{1,\dots,6\},\left\{(j,i)\middle| 1\leq i<j\leq 6\right\}\right).
\end{align*}
Then we have $c\mathscr{A}(\bm{n},G)^X=c\mathscr{A}(\bm{n}^{\prime},G^{\prime})$.
Therefore it is enough to prove that $c\mathscr{A}(\bm{n},G)$ is not hereditarily free when $\ell=6$.
\begin{figure}[H]
 \centering
\caption{The digraph $G$ when $\ell=6$.}
\begin{tikzpicture}[node distance={15mm}, main/.style = {draw, circle}] 
\node[main] (1) {$1$}; \node[main] (2) [below left of=1] {$2$}; \node[main] (3) [below right of=1] {$3$}; \node[main] (4) [below of=2] {$4$}; \node[main] (5) [below of=3] {$5$}; \node[main] (6) [below right of=4] {$6$};
\draw[->,>=stealth] (2) -- (1); \draw[->,>=stealth] (3) -- (1); \draw[->,>=stealth] (4) -- (1); \draw[->,>=stealth] (5) -- (1); \draw[->,>=stealth] (6) -- (1);
\draw[->,>=stealth] (3) -- (2); \draw[->,>=stealth] (4) -- (2); \draw[->,>=stealth] (5) -- (2);\draw[->,>=stealth] (6) to [out=180,in=210,looseness=1.5] (2);
\draw[->,>=stealth] (4) -- (3); \draw[->,>=stealth] (5) -- (3);\draw[->,>=stealth] (6) to [out=0,in=330,looseness=1.5] (3);
\draw[->,>=stealth] (5) -- (4); \draw[->,>=stealth] (6) -- (4); \draw[->,>=stealth] (6) -- (5); 
\end{tikzpicture}
\end{figure}
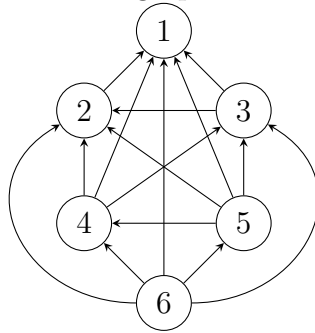

Let $\ell=6$, and let $H_1=\{x_1-x_4=n_4+\varepsilon_G(4,1)\}=\{x_1-x_4=n_4+1\}\in\mathscr{A}(\bm{n},G)$.
Then by writing $u=4$, we have that $V^{H_1}=\{2,3,4,5,6\}$, $E^{H_1}=\{(3,2),(5,2),(5,3),(5,4),(6,2),(6,3),$ $(6,4),(6,5)\}$, and $\bm{n}^{H_1}=(n_2,n_3,n_1+n_4+1,n_5,n_6)$.
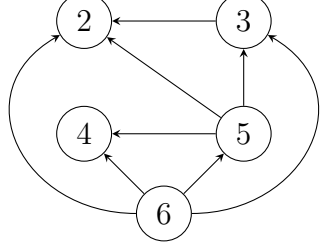
\begin{figure}[H]
 \centering
\caption{The digraph $G^{H_1}$.}
\begin{tikzpicture}[node distance={15mm}, main/.style = {draw, circle}] 
\node[main] (6) {$6$}; \node[main] (5) [above right of=6] {$5$}; \node[main] (4) [above left of=6] {$4$}; \node[main] (3) [above of=5] {$3$}; \node[main] (2) [above of=4] {$2$}; 
\draw[->,>=stealth] (3) -- (2); \draw[->,>=stealth] (5) -- (2);\draw[->,>=stealth] (6) to [out=180,in=210,looseness=1.5] (2);
\draw[->,>=stealth] (5) -- (3);\draw[->,>=stealth] (6) to [out=0,in=330,looseness=1.5] (3);
\draw[->,>=stealth] (5) -- (4); \draw[->,>=stealth] (6) -- (4); \draw[->,>=stealth] (6) -- (5);
\end{tikzpicture}
\end{figure}
Let $H_2=\{x_3-x_6=n_6+\varepsilon_G(6,3)\}=\{x_3-x_6=n_6+1\}\in\mathscr{A}(\bm{n}^{H_1},G^{H_1})$.
Then by writing $u=6$, we have $(V^{H_1})^{H_2}=\{2,4,5,6\}$, $(E^{H_1})^{H_2}=\{(5,2),(5,4),(6,2)\}$, and $(\bm{n}^{H_1})^{H_2}=(n_2,n_1+n_4+1,n_5,n_3+n_6+1)$.
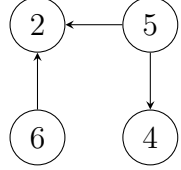
\begin{figure}[H]
 \centering
\caption{The digraph $(G^{H_1})^{H_2}$.}
\label{fig-digr-free-not-hered-free}
\begin{tikzpicture}[node distance={15mm}, main/.style = {draw, circle}] 
\node[main] (1) {$2$}; \node[main] (2) [right of=1] {$5$}; \node[main] (3) [below of=1] {$6$}; \node[main] (4) [below of=2] {$4$};
\draw[->,>=stealth] (2) -- (1); \draw[->,>=stealth] (3) -- (1); \draw[->,>=stealth] (2) -- (4);
\end{tikzpicture} 
\end{figure}
Finally let
\begin{align*}
H_3=
\begin{cases}
\{x_5-x_2=0\}\in\mathscr{A}\left((\bm{n}^{H_1})^{H_2},(G^{H_1})^{H_2}\right)&\quad{\rm if}\quad n_5\geq n_2,\\
\{x_5-x_2=n_2\}\in\mathscr{A}\left((\bm{n}^{H_1})^{H_2},(G^{H_1})^{H_2}\right)&\quad{\rm if}\quad n_5< n_2.
\end{cases}
\end{align*}
In any case, by writing $u=5$, we have $((V^{H_1})^{H_2})^{H_3}=\{4,5,6\}$ and $((E^{H_1})^{H_2})^{H_3}=\{(5,4),(6,5)\}$.
In addition, we have
\begin{align*}
((\bm{n}^{H_1})^{H_2})^{H_3}=
\begin{cases}
(n_1+n_4+1,n_5,n_3+n_6+1)\ &{\rm if}\ n_5\geq n_2,\\
(n_1+n_4+1,n_2+n_5,n_3+n_6+1)\ &{\rm if}\ n_5<n_2.
\end{cases}
\end{align*}
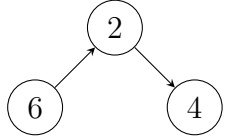
\begin{figure}[H]
 \centering
\caption{The digraph $((G^{H_1})^{H_2})^{H_3}$.}
\begin{tikzpicture}[node distance={15mm}, main/.style = {draw, circle}] 
\node[main] (1) {$2$}; \node[main] (2) [below left of=1] {$6$}; \node[main] (3) [below right of=1] {$4$};
\draw[->,>=stealth] (2) -- (1); \draw[->,>=stealth] (1) -- (3);
\end{tikzpicture} 
\end{figure}
By Lemma \ref{lem-A(n,G)-free-A1A2-l=3}, $c\mathscr{A}(((\bm{n}^{H_1})^{H_2})^{H_3},((G^{H_1})^{H_2})^{H_3})$ is not free.
This means that $c\mathscr{A}(\bm{n},G)$ is not hereditarily free.
\end{proof}

\noindent
{\it Remark.}
\begin{itemize}
\item[(1)] There exist total orders such that the digraphs $G^{H_1}$ and $(G^{H_1})^{H_2}$ satisfy the conditions (A1) and (A2). Therefore $c\mathscr{A}\left(\bm{n}^{H_1},G^{H_1}\right)$ and $c\mathscr{A}\left((\bm{n}^{H_1})^{H_2},(G^{H_1})^{H_2}\right)$ are free.
\item[(2)] The first counterexample of Orlik's conjecture is given by a $5$-arrangement with $21$ hyperplanes (which is found by Edelman and Reiner \cite{Edelman-Reiner-orlikconje}).
Now let $G$ be the same digraph as Figure \ref{fig-digr-free-not-hered-free}, and let $\overline{H}=\{x_5-x_2=0\}\in c\mathscr{A}(\bm{0},G)$.
Then $c\mathscr{A}(\bm{0},G)$ is free and $c\mathscr{A}(\bm{0},G)^{\overline{H}}$ is not free.
Since $|c\mathscr{A}(\bm{0},G)|=10$, there exists a free $4$-arrangement $\mathscr{A}_0$ with $10$ hyperplanes such that $c\mathscr{A}(\bm{0},G)=\Phi_1\times\mathscr{A}_0$.
We also have $\exp(\mathscr{A}_0)=\{1,3,3,3\}$ by Theorem \ref{thm-A(n,G)-free-A1A2}.
The arrangement $\mathscr{A}_0$ is a counterexample of Orlik's conjecture whose dimension and cardinality are less than the first counterexample.

At the same time, the $4$-arrangement with $10$ hyperplanes defined by
\begin{align}\label{eq-ideal-arr-count-Orlik-conj}
\left\{
\begin{array}{l}
\{x_1-x_2=0\},\{x_1-x_3=0\},\{x_2-x_3=0\},\{x_2+x_3=0\},\{x_1-x_4=0\},\\
\{x_1+x_4=0\},\{x_2-x_4=0\},\{x_2+x_4=0\},\{x_3-x_4=0\},\{x_3+x_4=0\}
\end{array}
\right\}
\end{align}
is also known to be a counterexample of Orlik's conjecture (see Example 3.2 and Remark 3.6 in \cite{Amend-Moller-Rohrle}).
After a change of coordinates, the defining polynomial of $\mathscr{A}_0$ coincides with that of the arrangement \eqref{eq-ideal-arr-count-Orlik-conj}.
\end{itemize}
\begin{Cor}
Let $\ell\geq 2$ and let $m$ be a nonnegative integer.
\begin{itemize}
\item[(1)] The cone $c\mathcal{C}_{\ell-1}(m)$ of the extended Catalan arrangement of type A is hereditarily inductively free.
\item[(2)] The cone $c\mathcal{S}_{\ell-1}(m)$ of the extended Shi arrangement of type A is hereditarily free if and only if $\ell\leq 5$.
\end{itemize}
\end{Cor}

\section*{Acknowledgments}
This work was supported by JSPS KAKENHI Grant Number JP20K14282.
The authors are grateful to Takahiro Nagaoka for helpful comments
for Theorem \ref{thm-free-A(n,G)-G-empty}.

\end{document}